\documentclass[reqno,11pt]{amsart}
\usepackage{setspace,tikz,xcolor,mathrsfs,listings,multicol}
\usepackage{rotating}
\usepackage[vcentermath]{youngtab}
\usepackage{fullpage}
\usepackage{enumerate}
\usepackage[all,cmtip]{xy}
\usetikzlibrary{arrows,matrix}
\usepackage{amssymb}
\usepackage[colorlinks=true, pdfstartview=FitV, linkcolor=blue, citecolor=blue, urlcolor=blue]{hyperref}

\newcommand{\g}{\mathfrak{g}}

\newcommand{\clfw}{\overline{\Lambda}} % classical fundamental weights
\newcommand{\inner}[2]{\left\langle #1, #2 \right\rangle}
\newcommand{\iso}{\cong}

\newcommand{\virtual}[1]{\widehat{#1}}

\newcommand{\induced}[1]{\widetilde{#1}} % induced map

\DeclareMathOperator{\wt}{wt} % weight
\DeclareMathOperator{\cl}{cl} % classical

\newcommand{\ZZ}{\mathbb{Z}}

\newcommand{\RR}{\mathbb{R}}

\definecolor{darkred}{rgb}{0.7,0,0} % darkred color
\newcommand{\defn}[1]{{\color{darkred}\emph{#1}}} % emphasis of a definition

\usepackage{listings}
\lstdefinelanguage{Sage}[]{Python}
{morekeywords={False,sage,True},sensitive=true}
\definecolor{dblackcolor}{rgb}{0.0,0.0,0.0}
\definecolor{dbluecolor}{rgb}{0.01,0.02,0.7}
\definecolor{dgreencolor}{rgb}{0.2,0.4,0.0}
\definecolor{dgraycolor}{rgb}{0.30,0.3,0.30}

\usepackage{xparse}

\makeatletter
\protected\def\specialmergetwolists{%
  \begingroup
  \@ifstar{\def\cnta{1}\@specialmergetwolists}
    {\def\cnta{0}\@specialmergetwolists}%
}
\def\@specialmergetwolists#1#2#3#4{%
  \def\tempa##1##2{%
    \edef##2{%
      \ifnum\cnta=\@ne\else\expandafter\@firstoftwo\fi
      \unexpanded\expandafter{##1}%
    }%
  }%
  \tempa{#2}\tempb\tempa{#3}\tempa
  \def\cnta{0}\def#4{}%
  \foreach \x in \tempb{%
    \xdef\cnta{\the\numexpr\cnta+1}%
    \gdef\cntb{0}%
    \foreach \y in \tempa{%
      \xdef\cntb{\the\numexpr\cntb+1}%
      \ifnum\cntb=\cnta\relax
        \xdef#4{#4\ifx#4\empty\else,\fi\x#1\y}%
        \breakforeach
      \fi
    }%
  }%
  \endgroup
}
\makeatother

\usepackage{xparse}
\DeclareDocumentCommand\rpp{ m m g }{
	\foreach \x [count=\s from 1] in {#1}{
	        {\ifnum\s=1
	                \draw (0,-\s)--(\x,-\s);
	                \fi}
	   \draw (0,-\s-1) to (\x,-\s-1);
	   \foreach \y in {0, ..., \x} {\draw (\y,-\s)--(\y,-\s-1);}
	}
	\specialmergetwolists{/}{#1}{#2}\ziplist
	\foreach \x/\y [count=\yi from 1] in \ziplist{
	    \node[anchor=west,font=\small] at (\x,-\yi - .5) {$\y$};
	}
	\IfValueT {#3}
	{\foreach \z [count=\zi from 1] in {#3} {\node[anchor=east,font=\small] at (0,-\zi - .5) {$\z$};}}
	{}
}

\theoremstyle{plain}
\newtheorem{thm}{Theorem}[section]
\newtheorem{lemma}[thm]{Lemma}
\newtheorem{conj}[thm]{Conjecture}
\newtheorem{prop}[thm]{Proposition}

\theoremstyle{definition}
\newtheorem{dfn}[thm]{Definition}
\newtheorem{ex}[thm]{Example}
\newtheorem{remark}[thm]{Remark}
\numberwithin{equation}{section}
\usepackage[colorinlistoftodos]{todonotes}

\begin{document}
\title{Virtualization map for the Littelmann path model}

\author[J.~Pan]{Jianping Pan}
\address[J. Pan]{Department of Mathematics, University of California, One Shields
Avenue, Davis, CA 95616}
\email{jppan@math.ucdavis.edu}

\author[T.~Scrimshaw]{Travis Scrimshaw}
\address[T. Scrimshaw]{School of Mathematics, University of Minnesota, 200 Oak St. SE, Minneapolis, MN 55455}
\email{tcscrims@gmail.com}
\urladdr{https://sites.google.com/view/tscrim/home}


\thanks{TS was partially supported by NSF grant OCI--1147247 and RTG grant NSF/DMS-1148634.}

\maketitle

\begin{abstract}
We show the natural embedding of weight lattices from a diagram folding is a virtualization map for the Littelmann path model, which recovers a result of Kashiwara.
As an application, we give a type independent proof that certain Kirillov--Reshetikhin crystals respect diagram foldings, which is a known result on a special case of a conjecture given by Okado, Schilling, and Shimozono.
\end{abstract}

\section{Introduction}
\label{sec:introduction}

In~\cite{K96}, Kashiwara gave a construction to realize a highest weight crystal~\cite{K90, K91} $B(\lambda)$ as a natural subset of $B(m\lambda)$ by dilating the crystal operators by $m$. Furthermore, Kashiwara gave necessary criteria for a generalization by realizing a $U_q(\g)$-crystal inside of a $U_q(\virtual{\g})$-crystal via considering a diagram folding of type $\virtual{\g}$ onto type $\g$. This realization and the corresponding isomorphism is known as a virtual crystal and virtualization map (the latter is also known as a similarity map) respectively.

Virtual crystals have been used effectively to reduce problems into simply-laced types~\cite{baker2000, OSS03III, OSS03II, SS2015, SS15, Scrimshaw15}, where it is typically easier to prove certain properties. Most notably, there is a set of axioms, known as the Stembridge axioms~\cite{Stembridge03}, which determines whether or not a crystal arises from a representation. In contrast, the only known (local) axioms are for type $B_2$~\cite{DKK09, Sternberg07}.

While there are numerous models for crystals, see, e.g.,~\cite{GL05, Kamnitzer10, KN94, LP08, L95-2, Nakajima03, SS2015}, many of them have not had their behavior under virtualization studied. Virtualization of the tableaux model was studied in~\cite{baker2000, OSS03III, OSS03II, SS15}, where the proofs were type-dependent and often involved tedious calculations. However, the situation is very different in other models. For rigged configurations, the virtualization map acts in a natural fashion~\cite{OSS03II, SS2015, SS2015II, SS15}. Additionally, the virtualization map for the polyhedral realization~\cite{Nakashima99,NZ97}, a semi-infinite tensor product of certain abstract crystals $\mathcal{B}_i$, is also well-behaved and is the setting in which Kashiwara proved his criteria~\cite{K96} for similarity.

The goal of this note is to describe the virtualization map for the Littelmann path model~\cite{L95-2}, which is given by paths in the real weight space. We show that the virtualization map is the induced map on the weight spaces. This map is natural as it reflects the fact that the Littelmann path model is based upon the geometry of the root system. As a result, we give precise conditions for the existence of this type of virtualization map, recovering Kashiwara's criteria. In addition, as the alcove model~\cite{LP08} is a discrete version of the Littelmann path model and Lakshmibai--Seshadri (LS) galleries~\cite{GL05} are based on the root system geometry, we expect a similar natural virtualization maps for these models.
Furthermore, to emphasize the relationship with the underlying geometry, it was shown in~\cite{JS15, NS08III} that the crystal structure on MV polytopes~\cite{Kamnitzer07,Kamnitzer10} admits natural virtualization maps.

We remark that this work partially overlaps with the work of Naito and Sagaki on LS paths invariant under a diagram automorphism~\cite{NS01,NS05II,NS10}. In their case, they only consider the fixed point subalgebras and allow automorphisms where there may be adjacent vertices in an orbit (e.g., the middle edge in the natural type $A_{2n}$ folding). Whereas in our case, we allow for arbitrary scaling factors and the vertices in the orbits of our automorphisms must not be adjacent. For example, this allows us to consider types $D_{n+1}^{(1)}$, $A_{2n}^{(2)}$, $A_{2n}^{(2)\dagger}$ as foldings of type $A_{2n-1}^{(2)}$.

The main application of our construction is about a particular class of finite crystals of affine type called Kirillov--Reshetikhin (KR) crystals~\cite{FOS09}. KR crystals and their associated KR modules are known to have many fascinating properties. KR modules are classified by their Drinfel'd polynomials and correspond to the minimial affinization of $B(r\clfw_r)$~\cite{CP95III, CP95II, CP95, CP96, CP96II, CP98}. Their characters (resp. $q$-characters~\cite{Nakajima03}) satisfy the Q-system (resp. T-system) relations~\cite{Chari01, Hernandez10}.
The graded characters of (resp. Demazure submodules of) tensor products of KR modules are (resp. Nonsymmetric) Macdonald polynomials at $t = 0$~\cite{LNSSS14, LNSSS14II} (resp.~\cite{LNSSS15}).
KR modules conjecturally admit crystal bases~\cite{HKOTY99, HKOTT02} and is known for most cases~\cite{JS10, KMOY07, LNSSS14, LNSSS14II, OS08, Yamane98}. KR crystals are (conjecturally) generally perfect~\cite{FOS10, KMOY07, Yamane98}, a technical property that allows them to be used to construct the Kyoto path model~\cite{KKMMNN92}, and known to be simple crystals in non-exceptional types~\cite{Okado13}.
KR crystals are also used to construct soliton cellular automata, where we refer the reader to~\cite{IKT12} and references therein. Tensor products of KR crystals are also (conjecturally) in bijection with rigged configurations, which arose from the study of the Bethe Ansatz of Heisenberg spin chains~\cite{KKR86, KR86, OSS13, OSSS16, OSS03, OSS03III, OSS03II, SS15, SS2006, Scrimshaw15}.

Tensor products of KR crystals of the form $\bigotimes_{i=1}^N B^{r_i,1}$ were constructed by Naito and Sagaki using projected level-zero LS paths in~\cite{NS03, NS06, NS08II, NS08} and using quantum LS paths and the quantum alcove model in~\cite{LNSSS14,LNSSS14II}. It was conjectured that KR crystals are also well-behaved under virtualization~\cite{OSS03II} and has been proven in a variety of special cases in a type-dependent fashion~\cite{Okado13, OSS03III, OSS03II, SS15}. We use the projected level-zero construction of Naito and Sagaki to give partial uniform results of the aforementioned conjecture.

This note is organized as follows. In Section~\ref{sec:background}, we give the necessary background on crystals, the Littelmann path model, level-zero crystals, and Kirillov--Reshetikhin crystals. In Section~\ref{sec:results}, we prove our main result, that the map on weight lattices naturally induces a virtualization map. In Section~\ref{sec:applications}, we apply our main result to show special cases of the conjectural virtualization of KR crystals.

\bigskip
\noindent {\bf Acknowledgements} We would like to thank Ben Salisbury for useful comments on an early draft of this paper. We thank the anonymous referees for comments. This work benefited from computations in {\sc SageMath}~\cite{combinat, sage}.

%===========================================================
\section{Background}
\label{sec:background}

Let $\g$ be a (symmetrizable) Kac--Moody algebra with index set $I$, (generalized) Cartan matrix $A = (A_{ij})_{i,j \in I}$, fundamental weights $\{\Lambda_i \mid i \in I\}$, simple roots $\{\alpha_i \mid i \in I\}$, simple coroots $\{\alpha_i^{\vee} \mid i \in I\}$, weight lattice $P$, coweight lattice $P^{\vee}$, root lattice $Q$, and coroot lattice $Q^{\vee}$. Let $U_q(\g)$ be the corresponding quantum group, $\mathfrak{h}^*_{\RR} = \RR \otimes_{\ZZ} P$, and $\mathfrak{h}_{\RR} = \RR \otimes_{\ZZ} P^{\vee}$ be the corresponding dual space. Let $\inner{\cdot}{\cdot} \colon \mathfrak{h}^*_{\RR} \times \mathfrak{h}_{\RR} \to \RR$ denote the canonical pairing by evaluation, in particular $\inner{\alpha_i^{\vee}}{\alpha_j} = A_{ij}$ and $\inner{\alpha_i^{\vee}}{\Lambda_j} = \delta_{ij}$, where $\delta_{ij}$ is the Kronecker delta. Let $\{ s_i \mid i \in I \}$ denote the set of simple reflections on $P$, where $s_i(\lambda) = \lambda - \inner{\alpha_i^{\vee}}{\lambda} \alpha_i$. Let $P^+$ denote the positive weight lattice, i.e., all nonnegative linear combinations of the fundamental weights.

%%%%%
\subsection{Crystals}

An \defn{abstract $U_q(\g)$-crystal} is a non-empty set $B$ along with maps
\begin{align*}
e_i, f_i & \colon B \to B \sqcup \{0\},
\\ \wt & \colon B \to P,
\end{align*}
which satisfy the conditions:
\begin{enumerate} 
\item $\varphi_i(b) = \varepsilon_i(b) + \inner{\alpha_i^{\vee}}{\wt(b)}$,

\item $f_i b = b'$ if and only if $b = e_i b'$ for $b' \in B$,
\end{enumerate}
for all $b \in B$ and $i \in I$, where
\begin{align*}
\varepsilon_i(b) & = \max \{k \in \ZZ \mid e_i^k b \neq 0 \},
\\ \varphi_i(b) & = \max \{k \in \ZZ \mid f_i^k b \neq 0 \}.
\end{align*}
The maps $e_i$ and $f_i$ are known as the \defn{crystal operators}. We are restricting ourselves to \defn{regular crystals} in this note, and so our definition is less general than than the usual definition of crystals (see, e.g.,~\cite{K90, K91}).
Therefore the entire $i$-string through an element $b \in B$ can be given diagrammatically as
\[
e_i^{\varepsilon_i(b)} b \xrightarrow[\hspace{25pt}]{i} \cdots \xrightarrow[\hspace{25pt}]{i} e_i b \xrightarrow[\hspace{25pt}]{i} b \xrightarrow[\hspace{25pt}]{i} f_i b \xrightarrow[\hspace{25pt}]{i} \cdots \xrightarrow[\hspace{25pt}]{i} f_i^{\varphi_i(b)} b.
\]

Let $B_1$ and $B_2$ be abstract $U_q(\g)$-crystals. A \defn{crystal morphism} $\psi \colon B_1 \to B_2$ is a map $B_1 \sqcup \{0\} \to B_2 \sqcup \{0\}$ such that
\begin{enumerate}
\item $\psi(0) = 0$;

\item if $b \in B_1$ and $\psi(b) \in B_2$, then $\wt\bigl(\psi(b)\bigr) = \wt(b)$, $\varepsilon_i\bigl(\psi(b)\bigr) = \varepsilon_i(b)$, and $\varphi_i\bigl(\psi(b)\bigr) = \varphi_i(b)$ for all $i \in I$;

\item for $b, b' \in B_1$, $f_i b = b'$, and $\psi(b), \psi(b') \in B_2$, we have $\psi(f_i b) = f_i \psi(b)$ and $\psi(e_i b') = e_i \psi(b')$ for all $i \in I$.
\end{enumerate}

%%%%%
\subsection{Littelmann path model}

Let $\pi, \pi' \colon [0, 1] \to \mathfrak{h}^*_{\RR}$, and define an equivalence relation $\sim$ by saying $\pi \sim \pi'$ if there exists a piecewise-linear, nondecreasing, surjective, continuous function $\phi \colon [0, 1] \to [0, 1]$ such that $\pi = \pi' \circ \phi$. Let $\Pi$ be the set of all piecewise-linear functions $\pi \colon [0, 1] \to \mathfrak{h}^*_{\RR}$ such that $\pi(0) = 0$ modulo $\sim$. We call the elements of $\Pi$ \defn{paths}.

Let $\pi_1, \pi_2 \in \Pi$. Let $\pi = \pi_1 \ast \pi_2$ denote the concatenation of $\pi_1$ by $\pi_2$:
\[
\pi(t) := \begin{cases}
\pi_1(2t) & 0 \leq t \leq 1/2, \\
\pi_1(1) + \pi_2(2t - 1) & 1/2 < t \leq 1.
\end{cases}
\]

For $\pi \in \Pi$, define $s_i(\pi)$ as the path given by $\bigl(s_i(\pi)\bigr)(t) := s_i\bigl(\pi(t)\bigr)$. Define the path $\pi^{\vee}$ by $\pi^{\vee}(t) = \pi(1 - t) - \pi(1)$.

Define a function $H_{i,\pi} \colon [0, 1] \to \RR$ by
\[
t \mapsto \inner{\alpha_i^{\vee}}{\pi(t)},
\]
and so we can express any path $\pi \in \Pi$ as
\begin{equation}
\label{eq:expr_sum}
\pi(t) = \sum_{i \in I} H_{i,\pi}(t) \Lambda_i.
\end{equation}
Let $m_i(\pi) := \min \{H_{i,\pi}(t) \mid t \in [0, 1]\}$ denote the minimal value of $H_{i,\pi}$.

We want to define $e_i^{(k)}$, where $k \in \ZZ_{>0}$. If $m_i(\pi) \leq -k$, then fix $t_1 \in [0,1]$ minimal such that $H_{i,\pi}(t_1) = m_i(\pi)$ and let $t_0 \in [0,1]$ maximal such that $H_{i,\pi}(t) \geq m_i(\pi) + k$ for $t \in [0, t_0]$. Choose $t_0 = t_{(0)} < t_{(1)} < \cdots < t_{(r)} = t_1$ such that either
\begin{enumerate}
\item $H_{i,\pi}(t_{(j-1)}) = H_{i,\pi}(t_{(j)})$ and $H_{i,\pi}(t) \geq H_{i,\pi}(t_{(j-1)})$ for $t \in [t_{(j-1)}, t_{(j)}]$; 

\item or $H_{i,\pi}$ is strictly decreasing on $[t_{(j-1)}, t_{(j)}]$ and $H_{i,\pi}(t) \geq H_{i,\pi}(t_{(j)})$ for $t \leq t_{(j-1)}$.
\end{enumerate}
Set $t_{(-1)} := 0$ and $t_{(r+1)} := 1$, and denote by $\pi_{(j)}$ the path defined by
\[
\pi_{(j)}(t) := \pi\bigl(t_{(j-1)} + t \; (t_{(j)} - t_{(j-1)}) \bigr) - \pi(t_{(j-1)})
\]
for $i = 0, 1, \dotsc, r+1$.

\begin{dfn}
\label{def:e}
Fix some $k \in \ZZ_{>0}$. If $m_i(\pi) > -k$, then $e_i^{(k)} \pi = 0$. Otherwise,
\[
e_i^{(k)} \pi = \pi_{(0)} \ast \eta_{(1)} \ast \eta_{(2)} \ast \cdots \eta_{(r)} \ast \pi_{(r+1)},
\]
where $\eta_{(j)} = \pi_{(j)}$ if $H_{i,\pi}$ behaves as in~(1) and $\eta_{(j)} = s_i(\pi_{(j)})$ if $H_{i,\pi}$ behaves as in~(2) on $[t_{(j-1)}, t_{(j)}]$.
\end{dfn}

Next we want to define $f_i^{(k)}$, where $k \in \ZZ_{>0}$. Let $\overline{t}_0 \in [0, 1]$ be maximal such that $H_{i,\pi}(\overline{t}_0) = m_i(\pi)$. If $H_{i,\pi}(1) - m_i(\pi) \geq k$, then fix $\overline{t}_1 \in [\overline{t}_0, 1]$ minimal such that $H_{i,\pi}(t) \geq m_i(\pi) + k$ for $t \in [\overline{t}_1, 1]$. Choose $\overline{t}_0 = \overline{t}_{(0)} < \overline{t}_{(1)} < \cdots < \overline{t}_{(r)} = \overline{t}_1$ such that either
\begin{itemize}
\item[($\overline{1}$)] $H_{i,\pi}(\overline{t}_{(j-1)}) = H_{i,\pi}(\overline{t}_{(j)})$ and $H_{i,\pi}(t) \geq H_{i,\pi}(\overline{t}_{(j-1)})$ for $t \in [\overline{t}_{(j-1)}, \overline{t}_{(j)}]$; or

\item[($\overline{2}$)] $H_{i,\pi}$ is strictly increasing on $[\overline{t}_{(j-1)}, \overline{t}_{(j)}]$ and $H_{i,\pi}(t) \geq H_{i,\pi}(\overline{t}_{(j)})$ for $t \geq \overline{t}_{(j)}$.
\end{itemize}
Let $\overline{t}_{(-1)} := 0$ and $\overline{t}_{(r+1)} := 1$, and denote by $\overline{\pi}_{(j)}$ the path defined by
\[
\overline{\pi}_{(j)}(t) := \pi\bigl(\overline{t}_{(j-1)} + t \; (\overline{t}_{(j)} - \overline{t}_{(j-1)}) \bigr) - \pi(\overline{t}_{(j-1)})
\]
for $i = 0, 1, \dotsc, r+1$. It is clear that $\pi = \overline{\pi}_{(0)} \ast \overline{\pi}_{(1)} \ast \cdots \ast \overline{\pi}_{(r+1)}$.

\begin{dfn}
Fix some $k \in \ZZ_{>0}$. If $H_{i,\pi}(1) - m_i(\pi) < k$, then $f_i^{(k)} \pi = 0$. Otherwise,
\[
f_i^{(k)} \pi = \overline{\pi}_{(0)} \ast \overline{\eta}_{(1)} \ast \overline{\eta}_{(2)} \ast \cdots \overline{\eta}_{(r)} \ast \overline{\pi}_{(r+1)},
\]
where $\overline{\eta}_{(j)} = \overline{\pi}_{(j)}$ if $H_{i,\pi}$ behaves as in~($\overline{1}$) and $\overline{\eta}_{(j)} = s_i(\overline{\pi}_{(j)})$ if $H_{i,\pi}$ behaves as in~($\overline{2}$) on $[\overline{t}_{(j-1)}, \overline{t}_{(j)}]$.
\end{dfn}

\begin{ex}
Consider $\g$ of type $C_2$ and the highest weight vector $\pi \in B(3\Lambda_1+\Lambda_2)$. Thus we have
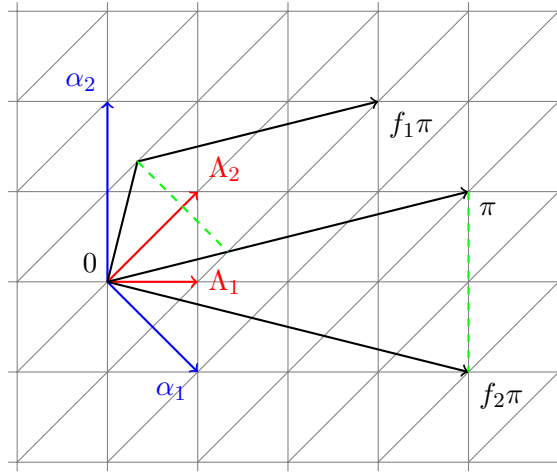
\[
\begin{tikzpicture}[scale=1.2]
\draw[step=1, gray, thin] (-1.1,-2.1) grid (5.1,3.1);
\draw[gray, thin] (-1,2) -- (0, 3);
\draw[gray, thin] (-1,1) -- (1, 3);
\draw[gray, thin] (-1,0) -- (2, 3);
\draw[gray, thin] (-1,-1) -- (3, 3);
\draw[gray, thin] (-1,-2) -- (4, 3);
\draw[gray, thin] (0,-2) -- (5, 3);
\draw[gray, thin] (1,-2) -- (5, 2);
\draw[gray, thin] (2,-2) -- (5, 1);
\draw[gray, thin] (3,-2) -- (5, 0);
\draw[gray, thin] (4,-2) -- (5, -1);
\draw (0, 0) node[anchor=south east]{0};
\draw[->, thick, red] (0, 0) -- (1,0) node[anchor=west] {$\Lambda_1$};
\draw[->, thick, red] (0, 0) -- (1,1) node[anchor=south west] {$\Lambda_2$};
\draw[->, thick, blue] (0, 0) -- (1,-1) node[anchor=north east] {$\alpha_1$};
\draw[->, thick, blue] (0, 0) -- (0,2) node[anchor=south east] {$\alpha_2$};
\draw[->, thick] (0,0) -- (4,1) node[anchor=north west] {$\pi$};
\draw[->, thick] (0,0) -- (1/3,4/3) -- (3, 2) node[anchor=north west] {$f_1 \pi$};
\draw[->, thick] (0,0) -- (4,-1) node[anchor=north west] {$f_2 \pi$};
\draw[dashed, green, thick] (4,1) -- (4,-1);
\draw[dashed, green, thick] (1/3,4/3) -- (4/3,1/3);
\end{tikzpicture}
\]
\end{ex}

\begin{remark}
If $k = 1$, we will simply write $e_i$ and $f_i$ for $e_i^{(1)}$ and $f_i^{(1)}$ respectively.
\end{remark}

\begin{lemma}[{\cite{K96}}]
\label{lemma:simple_scaling}
We have
\[
e_i^{(k)} = e_i^k \hspace{20pt} \text{ and } \hspace{20pt} f_i^{(k)} = f_i^k
\]
for all $i \in I$ and $k \in \ZZ_{>0}$.
\end{lemma}
Note that
\begin{equation}
\label{eq:crystal_duality}
f_i^{(k)}(\pi) = \bigl( e_i^{(k)} (\pi^{\vee}) \bigr)^{\vee}.
\end{equation}

\begin{thm}[\cite{K96, L95, L95-2}]
\label{thm:highest_weight_paths}
Let $B(\pi)$ denote the closure of $\pi$ under $e_i$ and $f_i$, for all $i \in I$.
The set $B(\pi)$ is a $U_q(\g)$-crystal.
Moreover, if $\pi$ is a path in the dominant chamber (i.e., $H_{i,\pi}(t) \geq 0$ for all $i \in I$ and $t \in [0, 1]$) such that $\pi(1) = \lambda \in P^+$, then $B(\pi)$ is isomorphic to the highest weight crystal of weight $\lambda$.
\end{thm}

For some path $\eta$, we say $B(\eta)$ has the \defn{integrality property} if for each $\pi \in B(\eta)$ and $i \in I$, the values $H_{i,\pi}(1)$ and all local minimums of $H_{i,\pi}(t)$ are integers.
We will also need the following fact.

\begin{lemma}[{\cite[Corollary~1]{L95-2}}]
\label{lemma:integrality}
Suppose $\eta$ denote a path in the dominant chamber. Then $B(\eta)$ has the integrality property.
\end{lemma}

Let $B(\lambda)$ denote $B(\pi)$, where $\pi$ is the straight line path from $0$ to $\lambda$, explicitly given by $\pi(t) = t \lambda$. If $\lambda \in P^+$, then $B(\lambda)$ is the set of all LS paths.

%%%%%
\subsection{Virtualization maps}

Let $A$ and $\virtual{A}$ be Cartan matrices with indexing sets $I$ and $\virtual{I}$ respectively. We call a map $\phi \colon \virtual{I} \to I$ a \defn{generalized diagram folding} if for all $j \neq j'$, we have $\phi(j) = \phi(j')$ implies $\virtual{A}_{j,j'} = 0$ (i.e., the nodes are non-adjacent). Consider a generalized diagram folding $\phi$ and scaling factors $(\gamma_i \in \ZZ_{>0})_{i \in \virtual{I}}$. Let $\Psi \colon \mathfrak{h}^*_{\RR} \to \virtual{\mathfrak{h}}^*_{\RR}$ be the map of the corresponding weight spaces given by
\begin{equation}
\label{eq:virtualization_map}
\Lambda_i \mapsto \sum_{j \in \phi^{-1}(i)} \gamma_j \virtual{\Lambda}_j.
\end{equation}

\begin{dfn}
Let $\virtual{B}$ be a $U_q(\virtual{\g})$-crystal and $V \subseteq \virtual{B}$. Let $\phi$ and $(\gamma_a \in \ZZ_{>0})_{a \in I}$ be the generalized diagram folding and the scaling factors. The \defn{virtual crystal operators} (of type $\g$) are defined as 
\begin{subequations}
\label{eq:virtual_crystal_ops}
\begin{align}
e^v_i & = \prod_{j \in \phi^{-1}(i)} \virtual{e}_j^{\gamma_j},
\\ f^v_i & = \prod_{j \in \phi^{-1}(i)} \virtual{f}_j^{\gamma_j}.
\end{align}
\end{subequations}
A \defn{virtual crystal} is a pair $(V, \virtual{B})$ such that $V$ has a $U_q(\g)$-crystal structure defined by
\begin{equation}
\label{eq:virtual_crystal}
\begin{aligned}
e_i & := e_i^v &&& f_i & := f_i^v,
\\ \varepsilon_i & := \gamma_j^{-1} \virtual{\varepsilon}_j &&& \varphi_i & := \gamma_j^{-1} \virtual{\varphi}_j,
\\ && \wt := \Psi^{-1} \circ \virtual{\wt}.
\end{aligned}
\end{equation}
for any $j \in \phi^{-1}(i)$.
\end{dfn}

\begin{remark}
\label{remark:commuting_ambient_ops}
We note that because the nodes in $\phi^{-1}(i)$ are non-adjacent, Equation~\eqref{eq:virtual_crystal_ops} is well-defined since all of the crystal operators of type $\virtual{\g}$ used commute.
\end{remark}

Suppose $B$ is isomorphic to the virtual crystal $(V, \virtual{B})$ (as $U_q(\g)$-crystals), then we call the isomorphism a \defn{virtualization map}.

%%%%%
\subsection{Level-zero and Kirillov--Reshetikhin crystals}
\label{sec:level0_kr_crystals}

In this section, we describe two classes of crystals of affine type that will be used in Section~\ref{sec:applications}: level-zero crystals and Kirillov--Reshetikhin crystals. For this section, let $\g$ be of affine type.

For type $\g$, let $0 \in I$ denote the special node, and let $I_0 := I \setminus \{0\}$ be the index set of the corresponding classical type $\g_0$. Let $\delta = \sum_{i \in I} a_i \alpha_i$ and $c = \sum_{i \in I} a_i^{\vee} \alpha_i^{\vee}$ denote the null root and the canonical central element of $\g$ respectively, where $a_i$ and $a_i^{\vee}$ are the Kac and dual Kac labels, respectively, as given in~\cite[Table~Aff.~1]{kac90}.

We say a weight $\lambda \in P$ is a \defn{level-zero weight} if $\lambda(c) = 0$. A level-zero weight is \defn{level-zero dominant} if $\inner{\lambda}{\alpha_i^{\vee}} \geq 0$ for all $i \in I_0$. The \defn{level-zero fundamental weights} $\{\varpi_i \in P \mid i \in I_0\}$ are defined as
\begin{equation}
\label{eq:level_zero_fund_wt}
\varpi_i = \Lambda_i - a_i^{\vee} \Lambda_0.
\end{equation}

Let $U_q'(\g) := U_q([\g, \g])$ be the quantum group corresponding to the derived subalgebra of $\g$. Define the \defn{classical projection} as $\cl \colon \mathfrak{h}^*_{\RR} \to \mathfrak{h}^*_{\RR} /  \RR\delta$ as the canonical projection. We identity the weight lattice of $U_q'(\g)$ with $P_{\cl} := \{ \cl(\lambda) \mid \lambda \in P\}$.

We now recall an important class of finite-dimensional $U_q'(\g)$-modules called \defn{Kirillov--Reshetikhin (KR) modules} denoted by $W^{r,s}$, where $r \in I_0$ and $s \in \ZZ_{>0}$, which are classified by their Drinfeld polynomials~\cite{CP95, CP98}. KR modules conjecturally admit a crystal bases~\cite{HKOTY99, HKOTT02}, which has been shown to exist for all non-exceptional types in~\cite{OS08} and in certain cases in exceptional types~\cite{JS10, KMOY07, LNSSS14, LNSSS14II, Yamane98}. The corresponding crystal to $W^{r,s}$ is known as a \defn{Kirillov--Reshetikhin (KR) crystal} and is denoted by $B^{r,s}$.
%KR crystals are known to have deep connections with mathematical physics and are typically \defn{perfect}~\cite{FOS10}, a technical condition which affords many nice representation theoretic properties (see, e.g.,~\cite{baker2000, HK02, KKMMNN92}).

KR crystals are also conjecturally well-behaved under the natural virtualization induced from the diagram folding $\phi$ given by the well-known natural embeddings of algebras~\cite{JM85}:
\begin{equation}
\label{eqn:affine_embeddings}
\begin{aligned}
C_n^{(1)}, A_{2n}^{(2)}, A_{2n}^{(2)\dagger}, D_{n+1}^{(2)} & \lhook\joinrel\longrightarrow A_{2n-1}^{(1)},
\\ B_n^{(1)}, A_{2n-1}^{(2)} & \lhook\joinrel\longrightarrow D_{n+1}^{(1)},
\\ E_6^{(2)}, F_4^{(1)} & \lhook\joinrel\longrightarrow E_6^{(1)},
\\ G_2^{(1)}, D_4^{(3)} & \lhook\joinrel\longrightarrow D_4^{(1)}.
\end{aligned}
\end{equation}

In this case, we define the scaling factors $(\gamma_j)_{j \in \virtual{I}}$ by $\gamma_j = \overline{\gamma}_i$ for all $j \in \phi^{-1}(i)$, where $\overline{\gamma}_i$ is given as follows.
\begin{enumerate}%[1.]
\item Suppose the Dynkin diagram of $\g$ has a unique arrow.
\begin{enumerate}%[(a)]
  \item Suppose the arrow points towards the component of the special node $0$. Then $\overline{\gamma}_i = 1$ for all $i \in I$.
  \item Otherwise, $\overline{\gamma}_i$ is the order of $\phi$ for all $i$ in the component of $0$ after removing the arrow and $\overline{\gamma}_i = 1$ in all other components.
\end{enumerate}
\item Otherwise the Dynkin diagram of $\g$ has 2 arrows and is a folding of type $A_{2n-1}^{(1)}$. Then $\overline{\gamma}_i = 1$ for all $1 \leq i \leq n-1$, and for $i \in \{0, n\}$, we have $\overline{\gamma}_i = 2$ if the arrow points away from $i$ and $\overline{\gamma}_i = 1$ otherwise.
\end{enumerate}

\begin{conj}{{\cite[Conj.~3.7]{OSS03II}}}
\label{conj:virtualization}
The KR crystal $B^{r,s}$ of type $\g$ virtualizes into
\[
\virtual{B}^{r,s} = \begin{cases}
B^{n,s} \otimes B^{n,s} & \text{if $\g = A_{2n}^{(2)}, A_{2n}^{(2)\dagger}$ and $r = n$,}
\\ \bigotimes_{r' \in \phi^{-1}(r)} B^{r', \overline{\gamma}_r s} & \text{otherwise.}
\end{cases}
\]
\end{conj}
Conjecture~\ref{conj:virtualization} was shown for $B^{r,1}$ in types $D_{n+1}^{(2)}$, $A_{2n}^{(2)}$, and $C_n^{(1)}$ in~\cite{OSS03III} and types $E_6^{(2)}$, $F_4^{(1)}$ (except $r=2$), $G_2^{(1)}$, and $D_4^{(3)}$ in~\cite{SS15}, $B^{1,s}$ for all non-exceptional types~\cite{OSS03II}, and $B^{r,s}$ for $r < n$ in types $B_n^{(1)}$ and $A_{2n-1}^{(2)}$~\cite{SS15}. The general case for non-exceptional affine types was done in~\cite{Okado13}.

We can extend $\cl$ to paths in a natural way by $\cl(\pi)(t) = \cl\bigl(\pi(t)\bigr)$. Now we can define the set of \defn{projected level-zero paths} $B(\lambda)_{\cl} := \{ \cl(\pi) \mid \pi \in B(\lambda) \}$. In~\cite{NS05}, it was shown that $B(\lambda)_{\cl}$ has a $U_q'(\g)$-crystal structure inherited from the $U_q(\g)$-crystal structure on $B(\lambda)$. We have the following key result that follows from~\cite{NS03, NS06, NS08II, NS08}.

\begin{thm}
\label{thm:projected_LS_paths}
Let $\g$ be of affine type. Let $\lambda = \sum_{i \in I_0} m_i \varpi_i$, with $m_i \in \ZZ_{\geq 0}$, is a dominant level-zero weight. Then there exists a canonical $U_q'(\g)$-crystal isomorphism
\[
\Upsilon \colon B(\lambda)_{\cl} \to \bigotimes_{i \in I_0} \left( B^{i,1} \right)^{\otimes m_i}.
\]
\end{thm}

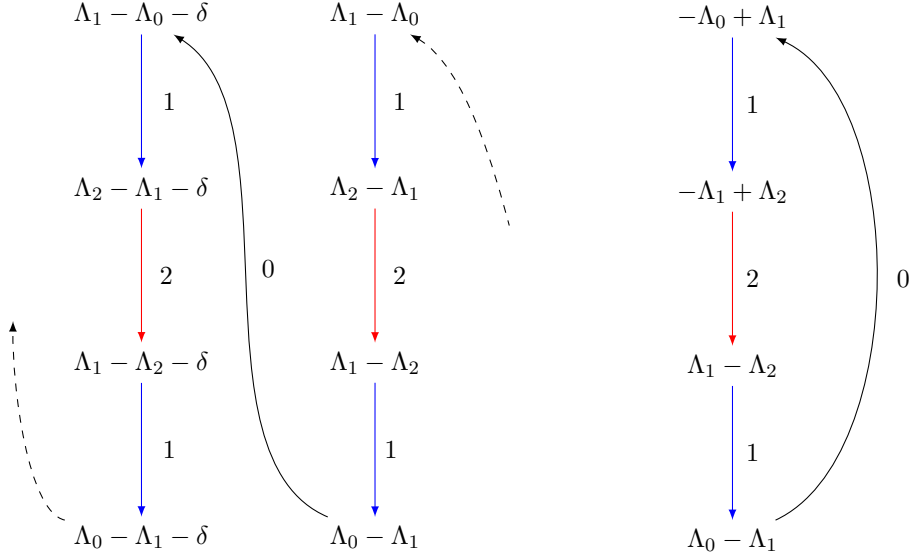
\begin{figure}

\[
\begin{tikzpicture}[>=latex,line join=bevel,xscale=1.1, yscale=0.9, every node/.style={scale=0.9},baseline=0]
  \node (node_8) at (24.0bp,9.5bp) [draw,draw=none] {$\Lambda_0-\Lambda_1-\delta$};
  \node (node_5) at (24.0bp,82.5bp) [draw,draw=none] {$\Lambda_1-\Lambda_2-\delta$};
  \node (node_4) at (24.0bp,155.5bp) [draw,draw=none] {$\Lambda_2-\Lambda_1-\delta$};
  \node (node_3) at (24.0bp,228.5bp) [draw,draw=none] {$\Lambda_1-\Lambda_0-\delta$};
  \node (node_6) at (104.0bp,82.5bp) [draw,draw=none] {$\Lambda_1-\Lambda_2$};
  \node (node_2) at (104.0bp,9.5bp) [draw,draw=none] {$\Lambda_0-\Lambda_1$};
  \node (node_1) at (104.0bp,228.5bp) [draw,draw=none] {$\Lambda_1-\Lambda_0$};
  \node (node_0) at (104.0bp,155.5bp) [draw,draw=none] {$\Lambda_2-\Lambda_1$};
  \definecolor{strokecol}{rgb}{0.0,0.0,0.0};
  \pgfsetstrokecolor{strokecol}
  \draw [blue,->] (node_6) ..controls (104.0bp,62.04bp) and (104.0bp,43.45bp)  .. (node_2);
  \draw (109.5bp,46.0bp) node {$1$};
  \draw [red,->] (node_4) ..controls (24.0bp,135.04bp) and (24.0bp,116.45bp)  .. (node_5);
  \draw (32.5bp,119.0bp) node {$2$};
  \draw [black,<-] (node_3) ..controls (78.218bp,186.7bp) and (39.443bp,42.9bp)  .. (node_2);
  \draw (67.5bp,122.0bp) node {$0$};
  \draw [blue,->] (node_5) ..controls (24.0bp,61.722bp) and (24.0bp,42.337bp)  .. (node_8);
  \draw (33.5bp,46.0bp) node {$1$};
  \draw [blue,->] (node_1) ..controls (104.0bp,204.04bp) and (104.0bp,185.45bp)  .. (node_0);
  \draw (112.5bp,192.0bp) node {$1$};
  \draw [blue,->] (node_3) ..controls (24.0bp,213.62bp) and (24.0bp,206.99bp)  .. (node_4);
  \draw (33.5bp,192.0bp) node {$1$};
  \draw [red,->] (node_0) ..controls (104.0bp,131.04bp) and (104.0bp,112.45bp)  .. (node_6);
  \draw (112.5bp,119.0bp) node {$2$};
  \draw [black,->,dashed] (node_8) ..controls (-15.0bp,20.0bp) and (-20.3bp,72.9bp)  .. (-20bp, 100bp);
  \draw [black,<-,dashed] (node_1) ..controls (135.2bp,206.7bp) and (145.4bp,160.9bp)  .. (150bp, 140bp);
  %\draw [green,->,dashed] (node_8) -- (node_2);
  %\draw [green,->,dashed] (node_5) -- (node_6);
  %\draw [green,->,dashed] (node_4) -- (node_0);
  %\draw [green,->,dashed] (node_3) -- (node_1);
  %\draw (62.5bp, 220.0bp) node {$\kappa$};
%
\end{tikzpicture}
\hspace{60pt}
\begin{tikzpicture}[>=latex,line join=bevel,scale=0.9, every node/.style={scale=0.9}]
  \node (node_2) at (104.0bp,155.5bp) [draw,draw=none] {$\Lambda_0-\Lambda_1$};
  \node (node_6) at (104.0bp,228.5bp) [draw,draw=none] {$\Lambda_1-\Lambda_2$};
  \node (node_1) at (104.0bp,374.5bp) [draw,draw=none] {$-\Lambda_0+\Lambda_1$};
  \node (node_0) at (104.0bp,301.5bp) [draw,draw=none] {$-\Lambda_1+\Lambda_2$};
  \definecolor{strokecol}{rgb}{0.0,0.0,0.0};
  \pgfsetstrokecolor{strokecol}
  \draw [blue,->] (node_6) ..controls (104.0bp,208.04bp) and (104.0bp,189.45bp)  .. (node_2);
  \draw (112.5bp,192.0bp) node {$1$};
  \draw [red,->] (node_0) ..controls (104.0bp,281.04bp) and (104.0bp,262.45bp)  .. (node_6);
  \draw (112.5bp,265.0bp) node {$2$};
  \draw [blue,->] (node_1) ..controls (104.0bp,354.04bp) and (104.0bp,335.45bp)  .. (node_0);
  \draw (112.5bp,338.0bp) node {$1$};
  \draw [black,<-] (node_1) ..controls (178.218bp,340.0bp) and (178.218bp,190.0bp)  .. (node_2);
  \draw (175.5bp,265.0bp) node {$0$};
\end{tikzpicture}
\]
\caption{A portion of the level-zero crystal $B(\Lambda_1 - \Lambda_0)$ (left) and $B^{1,1}$ constructed from $\Upsilon$ (right) in type $C_2^{(1)}$.}
\label{fig:C2}
\end{figure}

\begin{ex}
\label{ex:type_C2}
The construction of $B^{1,1}$ from the classical projection of $B(\Lambda_1 - \Lambda_0)$ is given in Figure~\ref{fig:C2}.
\end{ex}

%===========================================================
\section{Main results}
\label{sec:results}

In this section, we prove our main result.

\begin{thm}
\label{thm:virtualization}
Fix a map $\phi \colon \virtual{I} \to I$. The induced map $\Psi \colon P \to \virtual{P}$ given by Equation~\eqref{eq:virtualization_map} induces a virtualization map $\induced{\Psi} \colon B(\lambda) \to B\bigl(\Psi(\lambda) \bigr)$ given by
\[
\induced{\Psi}(\pi)(t) = \sum_{i \in I} H_{i,\pi}(t) \Psi(\Lambda_i)
\]
if and only if the following properties hold:
\begin{enumerate}[(I)]
\item For all $j \neq j' \in \phi^{-1}(i)$, we have $\virtual{A}_{j,j'} = 0$ (i.e., $\phi$ is a generalized diagram folding).

\item We have
\begin{equation}
\label{eq:virtual_simple_roots}
\Psi(\alpha_i) = \sum_{j \in \phi^{-1}(i)} \gamma_j \virtual{\alpha}_j.
\end{equation}
\end{enumerate}
\end{thm}

\begin{proof}
Fix some $\pi \in B(\lambda)$, and set
\[
\virtual{\pi}(t) := \induced{\Psi}(\pi)(t) = \sum_{i \in I}  H_{i,\pi}(t) \sum_{j \in \phi^{-1}(i)} \gamma_j \virtual{\Lambda}_j.
\]
It is clear that $\induced{\Psi}(\pi^{\vee}) = \bigl( \induced{\Psi}(\pi) \bigr)^{\vee}$. So $\induced{\Psi}$ is a virtualization map if and only if $\induced{\Psi}(e_i \pi) = e_i^v \induced{\Psi}(\pi)$ by Equation~\eqref{eq:crystal_duality}, where Equation~\eqref{eq:crystal_duality} also holds for the virtual crystal operators.
Note that $\varepsilon_i(\pi) = -\min_t H_{i,\pi}(t)$, and hence, we have $\virtual{\varepsilon}_j(\virtual{\pi}) / \gamma_j = \virtual{\varepsilon}_{j'}(\virtual{\pi}) / \gamma_{j'} = \varepsilon_i(\pi)$ for all $j, j' \in \phi^{-1}(i)$.
Thus note that property~(I) holds if and only if the virtual crystal operators are well-defined (see Remark~\ref{remark:commuting_ambient_ops}).

Fix some $i \in I$. We first consider the case when $e_i \pi = 0$. Thus we have $m_i(\pi) > -1$. By Lemma~\ref{lemma:integrality}, we can assume $m_i(\pi) \geq 0$. Next, take any $k \in \phi^{-1}(i)$. Then
\begin{align*}
m_k(\virtual{\pi}) & = \min \{ H_{k,\virtual{\pi}}(t) \mid t \in [0,1] \}
\\ & = \min \{ \inner{\virtual{\alpha}_k^{\vee}}{\virtual{\pi}(t)} \mid t \in [0,1] \}
\\ & = \min \{ \gamma_k H_{i,\pi}(t) \mid t \in [0,1] \}
\\ & = \gamma_k m_i(\pi) \geq 0.
\end{align*}
Therefore $\virtual{e}_k \virtual{\pi} = 0$ for any $k \in \phi^{-1}(i)$.

Now we consider the case when $e_i \pi \neq 0$. Thus, for any $j \in \phi^{-1}(i)$, we have
  \[
 H_{j,\virtual{\pi}}(t) = \inner{\virtual{\alpha}_j^{\vee}}{\induced{\Psi}(\pi)(t)} = \gamma_j H_{i,\pi}(t).
 \]
Fix some division of $\pi$ into subpaths
\[
e_i \pi = \pi_{(0)} \ast \eta_{(1)} \ast \eta_{(2)} \ast \cdots \ast \eta_{(r)} \ast \pi_{(r+1)}
\]
as in Definition~\ref{def:e}. Note that
\[
\induced{\Psi}(\pi) = \induced{\Psi}(\pi_{(0)}) \ast \induced{\Psi}(\pi_{(1)}) \ast \induced{\Psi}(\pi_{(2)}) \ast \cdots \induced{\Psi}(\pi_{(r)}) \ast \induced{\Psi}(\pi_{(r+1)})
\]
since conditions~(1) and~(2) for the subdivision of $\pi$ still hold when $H_{i,\pi}$ is scaled by a positive constant. Thus we have
\[
\induced{\Psi}(e_i \pi) = \induced{\Psi}(\pi_{(0)}) \ast \induced{\Psi}(\eta_{(1)}) \ast \induced{\Psi}(\eta_{(2)}) \ast \cdots \ast \induced{\Psi}(\eta_{(r)}) \ast \induced{\Psi}(\pi_{(r+1)}),
\]
and to show
\[
e_i^v \induced{\Psi}(\pi) = \virtual{\pi}_{(0)} \ast \virtual{\eta}_{(1)} \ast \virtual{\eta}_{(2)} \ast \cdots \ast \virtual{\eta}_{(r)} \ast \virtual{\pi}_{(r+1)} = \induced{\Psi}(e_i \pi),
\]
it is sufficient to show that $\virtual{\eta}_{(k)} = \induced{\Psi}(\eta_{(k)})$ for all $1 \leq k \leq r$ since $\virtual{\pi}_{(0)} = \induced{\Psi}(\pi_{(0)})$ and $\virtual{\pi}_{(r+1)} = \induced{\Psi}(\pi_{(r+1)})$.

Fix some $1 \leq k \leq r$. If $\eta_{(k)}=\pi_{(k)}$, then $H_{i,\pi}(t_{(k-1)}) = H_{i,\pi}(t_{(k)})$ and $H_{i,\pi}(t) \geq H_{i,\pi}(t_{(k-1)})$ for $t \in [t_{(k-1)}, t_{(k)}]$, i.e., condition (1) is satisfied. So for any $j \in \phi^{-1}(i)$ and $t \in [t_{(k-1)},t_{(k)}]$, we obtain
\[
H_{j,\virtual{\pi}}(t_{(k)}) = \gamma_j H_{i,\pi}(t_{(k)})=\gamma_j H_{i,\pi}(t_{(k-1)}) = H_{j,\virtual{\pi}}(t_{(k-1)}),
\]
and 
\[
H_{j,\virtual{\pi}}(t) = \gamma_j H_{i,\pi}(t) \geq  \gamma_j H_{i,\pi}(t_{(k-1)}) = H_{j,\virtual{\pi}}(t_{(k-1)})
\]
since $\gamma_i $ is positive. Therefore $\virtual{\eta}_{(k)} = \induced{\Psi}(\pi_{(k)}) = \induced{\Psi}(\eta_{(k)})$.

Now suppose $\eta_{(k)} = s_i \pi_{(k)}$, which implies $H_{i,\pi}$ is strictly decreasing on $[t_{(k-1)}, t_{(k)}]$ and $H_{i,\pi}(t) \geq H_{i,\pi}(t_{(k)})$ for $t \leq t_{(k-1)}$, i.e., condition~(2) is satisfied.
Then
\[
\eta_{(k)}(t) = (s_i \pi _{(k)})(t) = s_i\bigl( \pi_{(k)}(t) \bigr) = \pi_{(k)}(t) - \inner{\alpha_i^{\vee}}{\pi_{(k)}(t)} \alpha_i = \pi_{(k)}(t)-H_{i,\pi}(t)\alpha_i,
\]
and
\[
\induced{\Psi}(\eta_{(k)})(t) = \Psi(\pi_{(k)}(t) - H_{i,\pi}(t) \alpha_i) = \Psi\bigl( \pi_{(k)}(t) \bigr) - H_{i,\pi}(t) \Psi(\alpha_i).
\]

We note that from the definition of the crystal operators and Lemma~\ref{lemma:simple_scaling}, it is sufficient to consider $s_i^v = \prod_{j \in \phi^{-1}(i)} \virtual{s}_j$ for the action of $e_i^v$. Thus we have
\begin{align*}
s_i^{v}\left( \Psi\bigl(\pi_{(k)}(t)\bigr) \right) & =\prod_{j \in \phi^{-1}(i)} \virtual{s}_j \left( \Psi\bigl( \pi_{(k)}(t) \bigr) \right)
\\ &= \prod_{j \in \phi^{-1}(i)} \virtual{s}_j\left( \sum_{l \in I} H_{l,\pi}(t) \sum_{p \in \phi^{-1}(l)}\gamma_p \virtual{\Lambda}_p \right)
\\ &= \Psi (\pi_{(k)}(t))-\sum_{j \in \phi^{-1}(i)} \inner{\virtual{\alpha}_j^{\vee}}{\sum_{l \in I} H_{l,\pi}(t) \sum_{p \in \phi^{-1}(l)}\gamma_p \virtual{\Lambda}_p} \virtual{\alpha}_j \\
\\ &= \Psi\bigl( \pi_{(k)}(t) \bigr) - H_{i,\pi}(t)\sum_{j \in \phi^{-1}(i)} \gamma_j \virtual{\alpha}_j.
\end{align*}
Hence $s_i^{v}\bigl( \induced{\Psi}(\pi_{(k)}) \bigr) = \induced{\Psi}(\eta_{(k)})$ if and only if $\Psi(\alpha_i) = \sum_{j \in \phi^{-1}(i)} \virtual{\alpha}_j$. Therefore $\induced{\Psi}$ is a virtualization map if and only if Equation~\eqref{eq:virtual_simple_roots} holds.
\end{proof}

\begin{remark}
Theorem~\ref{thm:virtualization} holds for any path $\eta$ such that $B(\eta)$ (and subsequently $B(\virtual{\eta})$) has the integrality property.
\end{remark}

\begin{ex}
Consider the diagram folding from type $A_{3}$ to type $C_2$, where $\phi^{-1}(1)=\{1,3\}$ and $\phi^{-1}(2)=\{2\}$ and $\gamma_i = i$. The map $\Psi$ is given by
\begin{align*}
\Lambda_1 & \mapsto \virtual{\Lambda}_1 + \virtual{\Lambda}_3,
\\ \Lambda_2 & \mapsto 2\virtual{\Lambda}_2.
\end{align*}
Since $\alpha_i = \sum_{i' \in I} A_{i',i} \Lambda_{i'}$, we have
\begin{align*}
\virtual{\alpha}_1 + \virtual{\alpha}_3 & = (2 \virtual{\Lambda}_1 - \virtual{\Lambda}_2) + (2\virtual{\Lambda}_3 - \virtual{\Lambda}_2) = 2 (\virtual{\Lambda}_1 + \virtual{\Lambda}_3) - 2 \virtual{\Lambda}_2
\\ & = 2 \Psi(\Lambda_1) - \Psi(\Lambda_2)  = \Psi(\alpha_1),
\\ 2\virtual{\alpha}_2 & = 2 (2\virtual{\Lambda}_2 - \virtual{\Lambda}_1 - \virtual{\Lambda}_3) = 2(2\virtual{\Lambda}_2) - 2(\virtual{\Lambda}_1 + \virtual{\Lambda}_3)
\\ & = 2\Psi(\Lambda_2) - 2\Psi(\Lambda_1) = \Psi(\alpha_2).
\end{align*}
Therefore, the map $\induced{\Psi}$ is a virtualization map by Theorem~\ref{thm:virtualization}. In particular, if we consider the highest weight element $\pi(t) = 3t \Lambda_1 + t \Lambda_2 \in B(3\Lambda_1 + \Lambda_2)$, then
\begin{align*}
(f_2(\pi))(t) & = 5t \Lambda_1 - t \Lambda_2,
\\ \induced{\Psi}(\pi)(t) & = 3 t \Lambda_1 + 2 t \Lambda_2 + 3 t \Lambda_3.
\\ \induced{\Psi}(f_2 \pi)(t) & = \bigl(\virtual{f}_2^2 \induced{\Psi}(\pi) \bigr)(t) = 5 t \Lambda_1 - 2 t \Lambda_2 + 5 t \Lambda_3.
\end{align*}
\end{ex}

%===========================================================
\section{Applications}
\label{sec:applications}

In this section, we present an application of Theorem~\ref{thm:virtualization} to KR crystals.

\begin{prop}
\label{prop:projected_virtualization}
Let $\g$ be of affine type. Let $\induced{\Psi}$ be the virtualization map induced from the generalized diagram folding $\phi$ given in Section~\ref{sec:level0_kr_crystals}. Then there exists a $U_q'(\g)$-crystal virtualization map $\induced{\Psi}_{\cl}$ such that the diagram
\[
\xymatrix@C=4em@R=4em{B(\lambda) \ar[r]^{\induced{\Psi}} \ar[d]_{\cl} & B\bigl(\Psi(\lambda)\bigr) \ar[d]^{\cl} \\ B(\lambda)_{\cl} \ar[r]_{\induced{\Psi}_{\cl}} & B\bigl(\Psi(\lambda)\bigr)_{\cl}}
\]
commutes.
\end{prop}

\begin{proof}
Note that $\Psi(\delta) = a_0 \overline{\gamma}_0 \virtual{\delta}$, and that $\cl$ is given by quotienting by $\delta$.
Define $\Psi_{\cl} \colon P_{\cl} \to \virtual{P}_{\cl}$ by
\[
\Lambda_i \mapsto \overline{\gamma}_i \sum_{j \in \phi^{-1}(i)} \virtual{\Lambda}_j,
\]
and it is straightforward to verify that
\[
\alpha_i \mapsto \overline{\gamma}_i \sum_{j \in \phi^{-1}(i)} \virtual{\alpha}_j.
\]
Recall that the computation of the crystal operators does not depend on the coefficient of $\delta$.
Hence, we define $\induced{\Psi}_{\cl} \colon B(\lambda)_{\cl} \to B\bigl(\Psi(\lambda)\bigr)_{\cl}$ as the crystal map induced from $\Psi_{\cl}$, and the claim follows.
\end{proof}

\begin{thm}
\label{thm:virtualization_single_column}
Let $\g$ be of affine type. Suppose $r \in I$ is such that $\overline{\gamma}_r = 1$ or $\g$ is of type $A_{2n}^{(2)}$, $A_{2n}^{(2)\dagger}$. Then Conjecture~\ref{conj:virtualization} holds for $s = 1$.
\end{thm}

\begin{proof}
From Theorem~\ref{thm:virtualization}, there exists a virtualization map $\induced{\Psi} \colon B(\varpi_r) \to B\bigl(\Psi(\varpi_r) \bigr)$ from the diagram folding $\phi$ given in Section~\ref{sec:level0_kr_crystals}.  From Proposition~\ref{prop:projected_virtualization} and Theorem~\ref{thm:projected_LS_paths}, the result follows.
\end{proof}

Note that we require $\overline{\gamma}_r = 1$ as
\[
B\bigl(\Psi(\varpi_r)\bigr)_{\cl} \iso \bigotimes_{r' \in \phi^{-1}(r)} (B^{r',1})^{\otimes \overline{\gamma}_r},
\]
which does not agree with Conjecture~\ref{conj:virtualization} when $\gamma_r > 1$.

\begin{ex}
Consider type $C_2^{(1)}$ and $B^{1,1}$ from Example~\ref{ex:type_C2}.
Recall that for the folding of type $A_3^{(1)}$ onto $C_2^{(1)}$, we have $\overline{\gamma}_0 = \overline{\gamma}_2 = 1$ and $\overline{\gamma}_1$.
Applying $\Upsilon$ to the image of $B(\Lambda_1 - \Lambda_0)$ in $B(\virtual{\Lambda}_1 + \virtual{\Lambda}_3 - 2\virtual{\Lambda}_0)$ results in the virtual crystal
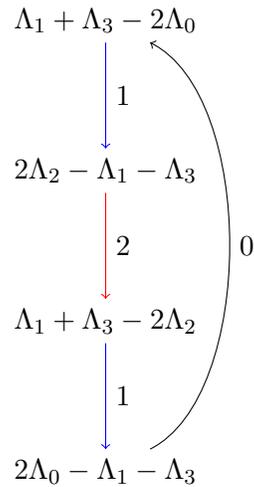
\[
\begin{tikzpicture}[yscale=2]
\node (1) at (0,0) [draw,draw=none] {$\Lambda_1+\Lambda_3-2\Lambda_0$};
\node (2) at (0,-1) [draw,draw=none] {$2\Lambda_2-\Lambda_1-\Lambda_3$};
\node (3) at (0,-2) [draw,draw=none] {$\Lambda_1+\Lambda_3-2\Lambda_2$};
\node (4) at (0,-3) [draw,draw=none] {$2\Lambda_0-\Lambda_1-\Lambda_3$};
\draw[blue,->] (1) -- (2) node[midway,right,black] {$1$};
\draw[red,->] (2) -- (3) node[midway,right,black] {$2$};
\draw[blue,->] (3) -- (4) node[midway,right,black] {$1$};
\draw[black,->] (4) ..controls(2,-2.5) and (2,-0.5) .. (1) node[midway,right,black] {$0$};
\end{tikzpicture}
\]
\end{ex}

%===========================================================
\appendix
\section{Examples with Sage}

We give an example of the virtualization map of $B(\Lambda_1)$ of type $C_2$ into $B(\virtual{\Lambda}_1 + \Lambda_3)$ of type $A_3$ using Sage~\cite{sage}.

\begin{lstlisting}
sage: PC = RootSystem(['C',2]).weight_space()
sage: LaC = PC.fundamental_weights()
sage: BC = crystals.LSPaths(LaC[1])
sage: PA = RootSystem(['A',3]).weight_space()
sage: LaA = PA.fundamental_weights()
sage: BA = crystals.LSPaths(LaA[1]+LaA[3])
sage: gens = BA.module_generators
sage: psi = BC.crystal_morphism(gens, codomain=BA)
sage: for x in BC:
....:     print "C2:", x
....:     print "A3:", psi(x)
....: 
C2: (Lambda[1],)
A3: (Lambda[1] + Lambda[3],)
C2: (-Lambda[1] + Lambda[2],)
A3: (-Lambda[1] + 2*Lambda[2] - Lambda[3],)
C2: (Lambda[1] - Lambda[2],)
A3: (Lambda[1] - 2*Lambda[2] + Lambda[3],)
C2: (-Lambda[1],)
A3: (-Lambda[1] - Lambda[3],)
\end{lstlisting}
Next we explicitly show the virtual crystal operators act as desired.
\begin{lstlisting}
sage: mg = BA.highest_weight_vector(); mg
(Lambda[1] + Lambda[3],)
sage: x1 = mg.f_string([1,3]); x1
(-Lambda[1] + 2*Lambda[2] - Lambda[3],)
sage: x2 = x1.f_string([2,2]); x2
(Lambda[1] - 2*Lambda[2] + Lambda[3],)
sage: x3 = x2.f_string([1,3]); x3
(-Lambda[1] - Lambda[3],)
\end{lstlisting}

\bibliographystyle{alpha}
\bibliography{biject}{}
\end{document}